\documentclass{article}
\usepackage{graphicx} 
\usepackage{authblk}
\usepackage[utf8]{inputenc} 
\usepackage[T1]{fontenc}    
\usepackage{url}

\usepackage{booktabs}       
\usepackage{amsfonts}       
\usepackage{nicefrac}       
\usepackage{microtype}      
\usepackage{xcolor}         
\usepackage{amsthm}
\usepackage{amssymb}
\usepackage{amsmath}
\usepackage{amsfonts}
\usepackage{mathtools}

\newtheorem*{theorem*}{Theorem}

\DeclareRobustCommand{\stirling}{\genfrac\{\}{0pt}{}}

\usepackage[backend=biber, style=alphabetic, sorting=nyt]{biblatex}
\title{A simple proof of almost sure convergence\\for the largest singular value of a product of\\Gaussian matrices}
\date{August 2024}
\author{Thiziri Nait Saada$^*$}
\author{Alireza Naderi$^*$}
\affil{Mathematical Institute, University of Oxford}
\date{}

\begin{document}

\maketitle

\def\thefootnote{*}\footnotetext{Equal contribution.}\def\thefootnote{\arabic{footnote}}

\begin{abstract}
    Let $m \geq 1$ and consider the product of $m$ independent $n \times n$ matrices $\mathbf{W} = \mathbf{W}_1 \dots \mathbf{W}_m$, each $\mathbf{W}_{i}$ with i.i.d.~normalised $\mathcal{N}(0, n^{-1/2})$ entries. It is shown in \cite{Penson_2011} that the empirical distribution of the squared singular values of $\mathbf{W}$ converges to a deterministic distribution compactly supported on $[0, u_m]$, where $u_m \coloneqq \frac{{(m+1)}^{m+1}}{m^m}$. This generalises the well-known case of $m=1$, corresponding to the Marchenko-Pastur distribution for square matrices. Moreover, for $m=1$, it was first shown by \cite{Geman} that the largest squared singular value almost surely converges to the right endpoint (the so-called ``soft edge'') of the support, i.e. $s_1^2(\mathbf{W}) \xrightarrow{a.s.} u_1$. Herein, we present a proof for the general case $s_1^2(\mathbf{W}) \xrightarrow{a.s.} u_m$ for $m\geq 1$. Although we do not claim novelty for our result, the proof is simple and does not require familiarity with modern techniques of free probability.

\end{abstract}

\section{Introduction}

The analysis of the spectral properties of random matrices is generally approached from two distinct perspectives, each requiring its own set of techniques. The macroscopic approach examines the global behaviour of eigenvalues, focusing on their limiting ``bulk'' density as the matrix size grows, while the microscopic perspective zooms in on finer details, studying local fluctuations, eigenvalue spacing statistics, and behaviour near special regions such as the spectral ``edge''. Despite the differences in scale, universality phenomena arise in both macroscopic and microscopic analyses; however, the behaviour at the bulk and the edge can vary depending on the ensemble under consideration. For instance, the ensemble of random Markov matrices studied in \cite{Bordenave_2011} exhibits the same limiting bulk distribution as Ginibre matrices, yet its largest eigenvalue behaves differently, remaining constrained to $1$.

The product of $m$ independent random matrices with i.i.d. Gaussian entries has been studied extensively due to its usefulness in a broad range of applications. On the macroscopic level, the limiting spectral distribution of such an ensemble was first studied by Penson and \.{Z}yczkowski \cite{Penson_2011}, where they explicitly derive the density function for arbitrary integer $m \geq 1$, the $k^\text{th}$ moment of which is given by the Fuss-Catalan number $FC_{m}(k) \coloneqq \frac{1}{mk+1} \binom{mk+k}{k}$. In the case of a single matrix ($m=1$), one recovers Marchenko-Pastur distribution (for square matrices) with Catalan numbers $C(k) \coloneqq \frac{1}{k+1} \binom{2k}{k}$ as moments. On the microscopic level, Akemann, Ipsen, and Kieburg \cite{Akemann_2013} derived the joint probability density function of the eigenvalues of such matrices as a determinantal point process with a correlation kernel given in terms of Meijer G-functions. Their explicit formulae for all correlation functions at finite size enabled the study of microscopic properties such as hard edge \cite{kuijlaars2014singular}, bulk, and soft edge \cite{liu2016bulk} scaling limits.

Using the soft edge scaling limit, one can establish the almost sure convergence of the largest singular value through a concentration-of-measure argument. However, a more straightforward approach relies solely on the moments of the squared singular value distribution, as previously demonstrated for a single Ginibre matrix with i.i.d. entries. The first known proof of this result was given by Geman \cite{Geman}, followed by further refinements that progressively weakened the moment assumptions on the matrix entries, e.g., \cite{jonsson1982some, Bai_1988}. Ultimately, it was shown in \cite{Yin_1988} that a bounded fourth moment is a necessary condition. 

In this work, we build on Geman's approach. Rather than relying on an intricate analysis of the distribution near the soft edge, we establish the almost sure convergence directly through the moments of the squared singular value distribution at finite $n$, as also presented in \cite{Akemann_2013}. This proof is more accessible and employs a combinatorial argument that is interesting in its own right.

\section{Theorem and Its Proof}

\begin{theorem*}\label{theorem}
    Let $m\geq 1$. Consider  $\mathbf{W}_1, \dots, \mathbf{W}_m \in \mathbb{R}^{n \times n}$ be independent Gaussian matrices with i.i.d. $\mathcal{N}(0,n^{-1/2})$ entries. Then, almost surely, 
    \begin{equation}
        s_1^2(\mathbf{W}_1 \dots \mathbf{W}_m) \xrightarrow{n\to \infty} u_m.
    \end{equation}
\end{theorem*}

\begin{proof} We denote by $\mathbf{W}$ the product $\mathbf{W}_1 \dots \mathbf{W}_m$. That almost surely $$\varliminf_{n \to \infty} s_1^2(\mathbf{W}) \geq u_m$$ is straightforward. Suppose otherwise that $s_1^2(\mathbf{W}) = \max_{i} s_i^2(\mathbf{W}) < u_m$ for infinitely many $n$ and there will be a subsequence of empirical distributions whose supports are strictly contained in $[0, u_m]$, which is a contradiction. Now we must show that $$\varlimsup_{n \to \infty} s_1^2(\mathbf{W}) \leq u_m$$ almost surely. Fix $z > u_m$. Following Geman's strategy, we demonstrate $\mathbb{P}\big( \varlimsup_{n \to \infty} s_1^2(\mathbf{W}) \geq z \big)=0$ using Borel-Cantelli lemma. For any $k \geq 1$,
    \begin{align*}
    \mathbb{P}\Big(s_1^2(\mathbf{W}) \geq z\Big) = \mathbb{P}\Big(\big(\frac{s_1^2(\mathbf{W})}{z} \big)^k \geq 1 \Big) \leq \mathbb{E}\Big(\big(\frac{s_1^2(\mathbf{W})}{z} \big)^k\Big)
\end{align*}
by Markov's inequality.
Therefore, to exhibit an almost sure convergence, it suffices to show 
\begin{align*}
    \sum_{n} \mathbb{E}\Big(\big(\frac{s_1^2(\mathbf{W})}{z} \big)^k\Big) < \infty,
\end{align*}
for some $k$. We will shortly see how we can carefully choose $k=k_n$ to make the above sum converge. Namely, the remainder of the proof is devoted to establishing 

\begin{align}\label{WTS}
     \sum_{n} \mathbb{E}\Big(\big(\frac{s_1^2(\mathbf{W})}{z} \big)^{k_n}\Big) < \infty,
\end{align}
where, 
\begin{align*}
     k_n = \lceil{w \log n \rceil},
\end{align*}
and
\begin{align*}
     w>\frac{3}{\log(z/u_m)}.
\end{align*}

We may simply bound a term of the series in~\eqref{WTS} by the $k_n$-th moment of the empirical (non-limiting) distribution $\mu_n$ of the squared singular values of $\mathbf{W}$, i.e.
\begin{align*}
    \mathbb{E}\Big(\big(s_1^2(\mathbf{W}) \big)^{k_n}\Big) & \leq \mathbb{E}\Big(\sum_{i=1}^n \big(s_i^2(\mathbf{W}) \big)^{k_n}\Big) \\
    & = n \mathbb{E}\Big(\frac{1}{n}\sum_{i=1}^n \big(s_i^2(\mathbf{W}) \big)^{k_n}\Big)\\
    & = n G(m,n,k_n),
\end{align*}
where $G(m, n, k_n)$ is the $k_n$-th moment of $\mu_n$.
We borrow the computation of $G(m, n, k_n)$ from \cite[Eq.~(58)]{Akemann_2013}, where it is worked out for the general case of the product of rectangular Gaussian matrices.\footnote{In the case where all matrices are square, it is well known that $G(m,n,k)$ converges to the Fuss-Catalan number,
\begin{align*}
    \mathrm{FC}_m(k) \coloneqq \frac{1}{mk+1} \binom{mk+k}{k},
\end{align*}
for any fixed $k$ and $m$; see \cite{Penson_2011}.} The calculation from \cite{Akemann_2013} provides us with the following non-asymptotic formula, valid for any integer $k \geq 1$,
\begin{align*}
   n^{mk+1} G(m,n,k) = \sum_{i=0}^{n-1} \frac{(-1)^{1+i} \prod_{j=0}^{n-1} (j-k-i)}{i! (n-1-i)! k} \times \Big(\frac{\Gamma(k+i+1)}{\Gamma(i+1)} \Big)^m,
\end{align*}
which can be further simplified as,
\begin{align*}
   n^{mk+1} G(m,n,k) &= \frac{1}{k!} \sum_{i=n-k}^{n-1} (-1)^{n+1+i} \Big(\frac{(k+i)!}{i!} \Big)^{m+1} \binom{k-1}{k+i-n} \\
   &= \frac{1}{k!} \sum_{j=0}^{k-1} (-1)^{k+1-j} \Big(\frac{(n+j)!}{(n+j-k)!} \Big)^{m+1} \binom{k-1}{j} \\
   \begin{split} &= \frac{1}{k!} \sum_{j=0}^{k-1} (-1)^{k+1-j} \binom{k-1}{j} \\
 & \hspace{1.5cm} \times \big((n+j)(n+j-1)\dots (n+j-k+1) \big)^{m+1} \end{split}\\
   \begin{split}&= \frac{n^{k(m+1)}}{k!} \sum_{j=0}^{k-1} (-1)^{k+1-j} \binom{k-1}{j}\\
   & \hspace{1.5cm} \times \big((1+\frac{j}{n})(1-\frac{1}{n}+\frac{j}{n})\dots (1-\frac{k-1}{n}+\frac{j}{n}) \big)^{m+1}.
   \end{split}
\end{align*}

We now introduce $\beta_r$ as the coefficient of $x^r$ in the expansion of the polynomial $P(x) \coloneqq \prod_{i=0}^{k-1} (1-\frac{i}{n} + x)^{m+1}$. Let $\mathcal{R}$ be the multiset of the $k(m+1)$ roots of $P$ (counted with multiplicity), then each $\beta_r$ can be explicitly written as, 

\begin{align*}
    \beta_r = \sum_{\substack{S \subseteq \mathcal{R}\\|S| = k(m+1)-r}} \prod_{i \in S} (-i).
\end{align*}
Provided $k \leq n$, all roots of $P$ are negative with magnitude in $[1-\frac{k-1}{n}, 1]$. Therefore, for all $0 \leq r \leq k(m+1)$, we have
\begin{equation*}\label{upper_bound_beta}
   \binom{k(m+1)}{r} \big(1-\frac{k-1}{n}\big)^{k(m+1)-r} \leq \beta_r \leq \binom{k(m+1)}{r},
\end{equation*}
which yields the asymptotic equivalence\footnote{In the following, we write $f(n)\sim g(n)$ whenever $\lim_{n\to\infty} f(n)/g(n) = 1$.}
\begin{equation}\label{growth_beta}
    \beta_r \sim \binom{k(m+1)}{r},
\end{equation}
whenever $k^2 = o(n)$.
Substituting back into the last formula, we get 
\begin{align*}
   n^{mk+1} G(m,n,k) &= \frac{n^{k(m+1)}}{k!} \sum_{j=0}^{k-1} (-1)^{k+1-j} \binom{k-1}{j} \sum_{r=0}^{k(m+1)} \beta_r {\big(\frac{j}{n}\big)}^r   \\
   &= \frac{n^{k(m+1)}}{k!} \sum_{r=0}^{k(m+1)} n^{-r} \beta_r \sum_{j=0}^{k-1} (-1)^{k+1-j} \binom{k-1}{j} j^r.
\end{align*}
The above alternating sums are known to be equal to Stirling numbers of the second kind $\stirling{n}{k}$, defined as the number of ways to partition a set of $n$ objects into $k$ non-empty subsets. They are given by
\begin{equation*}
    \stirling{n}{k} = \frac{1}{k!}\sum_{i=0}^{k} (-1)^{k-i} \binom{k}{i} i^n .
\end{equation*}
Thus, we can rewrite our original quantity as
\begin{equation*}
    n^{mk+1} G(m,n,k) = \frac{n^{k(m+1)} (k-1)!}{k!} \sum_{r=0}^{k(m+1)} n^{-r} \beta_r \stirling{r}{k-1}.
\end{equation*}
By definition, it is clear that $\stirling{n}{k} = 0$ for $k>n$, therefore, 
\begin{align}\label{mid_point_sum}
    n^{mk+1} G(m,n,k) &= \frac{n^{k(m+1)}}{k} \sum_{r=k-1}^{k(m+1)} n^{-r} \beta_r \stirling{r}{k-1}.
\end{align}

We now move on to showing that for our specific choice of $k=k_n$, the latter sum is asymptotically dominated by its first term. To this end, we demonstrate that each term in the sum is dominated by the term immediately preceding it. It is sufficient to show, for any $k-1 \leq r\leq km+k-1$,

\begin{align}
    \frac{\beta_{r+1} \stirling{r+1}{k-1} n^{-(r+1)}}{\beta_{r} \stirling{r}{k-1} n^{-r}}< C n^{-\varepsilon},
\end{align}
for some $\varepsilon >0$ and $C$ independent of $n$.

First, observe the recurrence relation satisfied by Stirling numbers of the second kind, which immediately gives, 

\begin{align}\label{ratio_stirling}
    \frac{\stirling{r+1}{k-1}}{\stirling{r}{k-1}} = k-1 + \frac{\stirling{r}{k-2}}{\stirling{r}{k-1}}.
\end{align}
Among other properties, Stirling numbers of the second kind are shown in \cite{lieb1968concavity} to be logarithmically concave, meaning, 
\begin{align*}
    \stirling{n}{k}^2 \geq \stirling{n}{k+1} \stirling{n}{k-1},
\end{align*}
for any $k=2, \dots, n-1$. This directly implies that,
\begin{align*}
    \frac{\stirling{n}{k-1}}{\stirling{n}{k}} \leq \frac{\stirling{n}{k}}{\stirling{n}{k+1}} \leq \dots \leq \frac{\stirling{n}{n-1}}{\stirling{n}{n}} = \binom{n}{2}.
\end{align*}
Back to our problem, we can thus bound the ratio in Eq.~\eqref{ratio_stirling} as follows:
\begin{align*}
    \frac{\stirling{r+1}{k-1}}{\stirling{r}{k-1}} = k-1 + \binom{r}{2} \leq \frac{1}{2} r(r+1).
\end{align*}
To bound the ratio $\frac{\beta_{r+1}}{\beta_r}$, since $k_n = \lceil w \log(n) \rceil$ satisfies $k_n^2=o(n)$, we can use Eq.~\eqref{growth_beta} to write
\begin{align*}
    \frac{\beta_{r+1}}{\beta_r} \sim \frac{\binom{k(m+1)}{r+1}}{\binom{k(m+1)}{r}} = \frac{k(m+1) - r}{r+1} < m+1.
\end{align*}
Altogether, 
\begin{align*}
     \frac{n^{-(r+1)}\beta_{r+1} \stirling{r+1}{k-1}}{n^{-r}\beta_{r} \stirling{r}{k-1}} < n^{-1} \frac{m+1}{2}(r+1)^2.
\end{align*}
Therefore, if we restrict the growth of $r$ (and hence that of $k_n$) with $n$ such that $r+1 < n^{(1- \varepsilon)/2}$, we get
\begin{align}
     \frac{n^{-(r+1)}\beta_{r+1} \stirling{r+1}{k-1}}{n^{-r}\beta_{r} \stirling{r}{k-1}} < \frac{m+1}{2} n^{-\varepsilon},
\end{align}
which is sufficient to prove what needs to be shown. Our specific choice for $k_n = \lceil w \log(n) \rceil$ satisfies this constraint.

Having proven that each term in the sum in Eq.~\eqref{mid_point_sum} is dominated by the preceding one, this sum can be effectively approximated by its first term, i.e.

\begin{align*}
    n^{mk_n+1} G(m, n, k_n) & = \frac{n^{k_n(m+1)}}{k_n} \sum_{r=k_n-1}^{k_n(m+1)} n^{-r} \beta_r \stirling{r}{k_n-1} \\
    &= \frac{n^{k_n(m+1)}}{k_n} n^{-(k_n-1)} \beta_{k_n-1} \big(1 + o(1)\big) \\
    &= n^{mk_n+1} \frac{\beta_{k_n-1}}{k_n} \big(1 + o(1)\big),
\end{align*}
or simply $$G(m, n, k_n) = \frac{\beta_{k_n-1}}{k_n} \big(1 + o(1)\big).$$
Given that $k_n\to \infty$ as $n$ grows, we can use Stirling's approximation formula to get
\begin{align}\label{approx_beta}
    \frac{\beta_{k_n-1}}{k_n} & \sim \frac{1}{k_n} \binom{k_n(m+1)}{k_n - 1} = \frac{\big(k_n(m+1)\big)!}{k_n! (mk_n+1)!} \nonumber\\
    & \sim \frac{\sqrt{2\pi k_n(m+1)}}{\sqrt{2\pi k_n} \sqrt{2\pi (mk_n +1)}} \frac{\big(k_n(m+1)\big)^{k_n(m+1)}}{k_n^{k_n} (mk_n+1)^{k_nm+1}}\nonumber\\
    & \sim \sqrt{\frac{m+1}{2\pi m^3}} \frac{u_m^{k_n}}{k_n^{3/2}}.
\end{align}
Ultimately, 

\begin{align*}
    \mathbb{E}\Big(\big(s_1^2(\mathbf{W}) \big)^{k_n}\Big) & \leq n G(m,n,k_n) \\
    & = n\frac{\beta_{k_n-1}}{k_n} \big(1 + o(1)\big) \\
    & = \sqrt{\frac{m+1}{2\pi m^3}} \frac{n}{k_n^{3/2}} u_m^{k_n} \big(1 + o(1)\big).
\end{align*}

Substituting back in the Borel-Cantelli sum in~\eqref{WTS}, it remains to show that $$\sum_n \frac{n}{k_n^{3/2}} \Big(\frac{u_m}{z}\Big)^{k_n} <\infty$$ for our choice of $k_n = \lceil \frac{3}{\log(z/u_m)} \log n \rceil$. 
Precisely, since $z/u_m > 1$,
\begin{align*}
    \log \Bigg[ \frac{n}{k_n^{3/2}} \Big(\frac{u_m}{z}\Big)^{k_n} \Bigg] &= \log n - \frac{3}{2} \log k_n + k_n \log \big( \frac{u_m}{z} \big) \\
    & \leq \log n + k_n \log \big( \frac{u_m}{z} \big) \\
    & = \log n - k_n \log \big( \frac{z}{u_m} \big) \\
    & \leq \log n - \frac{3 \log n}{\log(z/u_m)} \log \big( \frac{z}{u_m} \big) \\
    & = -2 \log n,
\end{align*}
hence the series converges and the proof is complete. 
\end{proof}

\subsection*{Acknowledgement}
The authors thank Jon Keating for providing feedback on the manuscript.

\printbibliography
\end{document}